\documentclass[11pt]{amsart}
\usepackage{amsmath,amssymb,amsthm}
\usepackage[margin=1.15in]{geometry}
\usepackage{microtype}

\newtheorem{theorem}{Theorem}[section]
\newtheorem{lemma}[theorem]{Lemma}
\newtheorem{proposition}[theorem]{Proposition}
\newtheorem{corollary}[theorem]{Corollary}
\theoremstyle{definition}
\newtheorem{definition}[theorem]{Definition}
\newtheorem{example}[theorem]{Example}
\theoremstyle{remark}
\newtheorem{remark}[theorem]{Remark}
\newtheorem{question}[theorem]{Question}

\newcommand{\N}{\mathbb{N}}
\newcommand{\Z}{\mathbb{Z}}
\newcommand{\C}{\mathbb{C}}

\newcommand{\Prob}{\mathbb{P}}
\newcommand{\supp}{\operatorname{supp}}

\begin{document}

\title[Directional growth of lamplighter groups]{A phase transition in the directional growth\\ of lamplighter groups}

\author{Mohammad F. Marashdeh}
\address{Department of Mathematics, Mutah University, Karak, Jordan}
\email{marashdeh@mutah.edu.jo}

\begin{abstract}
Let $G$ be a finitely generated group with a surjective homomorphism $\pi\colon G\to\Z$. The directional growth spectrum $I(\beta)$ is the exponential rate of the number of elements of length at most $n$ lying over $\lfloor\beta n\rfloor$; its maximum is the growth rate. Wherever $I$ has been computed---free, hyperbolic and relatively hyperbolic groups, free abelian groups---it is strictly concave and real-analytic, as Perron--Frobenius theory dictates. We compute $I$ in closed form for the lamplighter groups $F\wr\Z$, $|F|=r+1$, with the standard generators, and obtain a different picture: $I$ is affine on $[-\beta^{*},\beta^{*}]$ and strictly concave beyond it, with a second-order transition at $\beta^{*}$. Elements conditioned to the affine phase backtrack macroscopically, with lamp density independent of $\beta$. The series $\sum_x s^{|x|}y^{\pi(x)}$ is rational, and the transition is an exchange of dominant singularities, the inner one independent of $y$. The peak of $I$ is $\log\omega$ with $\omega^{2}=\omega+r$, the central value $\log\rho$ with $\rho^{3}=\rho+r$; the base group is therefore co-amenable yet grows strictly slower, by an amount unbounded in $r$. For $r=1$ these constants are the golden ratio and the plastic number.
\end{abstract}

\subjclass[2020]{Primary 20F65; Secondary 20F69, 20E22, 05A16, 60F10}
\keywords{Growth of groups, lamplighter group, wreath product, relative growth, growth series, large deviations, phase transition}

\maketitle

\section{Introduction}\label{sec:intro}

For a finitely generated group $G$ with finite generating set $S$, the growth function $V(n)$ counts the ball of radius $n$, and its exponential rate $\omega=\lim V(n)^{1/n}$ is a fundamental invariant \cite{Gromov81,Grigorchuk84,dlH00}. This invariant does not capture the distribution of elements inside the ball. If $G$ admits a surjective homomorphism $\pi\colon G\to\Z$, the ball decomposes into fibers over $\pi$, and one may ask how many elements of length at most $n$ have height $\lfloor\beta n\rfloor$. The exponential rate of that count, as a function of $\beta$, is a refinement of the growth rate, which it recovers as its maximum.

This directional counting problem has been solved in several settings, all of them nonpositively curved or abelian. For free groups it was settled by Sharp \cite{Sharp01} by symbolic dynamics, following Rivin's central limit theorem \cite{Rivin10}; for relatively hyperbolic groups and graph products by Gekhtman, Taylor and Tiozzo \cite{GTT20}; and for free abelian groups by Duchin, Leli\`evre and Mooney \cite{DLM12}. In every one of these cases the answer is a strictly concave, real-analytic function of $\beta$. This is not an accident of the examples: such counting problems are governed by transfer operators, and the analyticity and strict concavity of the resulting rate function are exactly what Perron--Frobenius theory delivers.

For amenable groups of exponential growth no such computation appears in the literature, and the transfer-operator machinery is unavailable. We carry the computation out for the lamplighter groups, in closed form and for every finite lamp group, by elementary means: word length in these groups obeys an exact formula, which turns fiber counting into a two-parameter binomial sum whose Laplace asymptotics can be solved explicitly. The answer is qualitatively different from the curved cases. The spectrum is affine on a symmetric interval about the origin and strictly concave outside it, with a genuine second-order phase transition at the boundary; the affine phase reflects a regime in which optimal elements overshoot their target and return, and it has no counterpart in the settings listed above. Along the way the two critical exponents appear as roots of a cubic and a quadratic; the growth series refined by $\pi$ is computed in closed rational form, which locates the transition in its singularities; and the fiber over the identity is seen to grow strictly slower than the group, which bears on the growth-gap criteria for co-amenable subgroups.

\begin{definition}\label{def:spectrum}
Let $G$ be generated by the finite symmetric set $S$, let $\pi\colon G\to\Z$ be a surjective homomorphism with $\pi(S)\subset\{-1,0,1\}$, and for $j\in\Z$, $n\in\N$ set
\[
N_j(n)\;=\;\#\{x\in G:\ |x|\le n,\ \pi(x)=j\}.
\]
The \emph{directional growth spectrum} of $(G,S,\pi)$ is
$I(\beta)=\lim_{n\to\infty}\tfrac1n\log N_{\lfloor\beta n\rfloor}(n)$, $\beta\in[-1,1]$, whenever the limit exists.
\end{definition}

Since $\pi$ is $1$-Lipschitz, $N_j(n)=0$ for $|j|>n$, and $\sum_{|j|\le n}N_j(n)=V(n)$; hence $\sup_\beta I(\beta)=\log\omega(G,S)$ whenever $I$ exists and is continuous.

Fix a finite group $F$ with $|F|=r+1\ge2$ and let $L_F=F\wr\Z$ carry the standard generating set $S=\{t^{\pm1}\}\cup(F\setminus\{e\})$: the generator $t$ moves the cursor, and each $c\in F\setminus\{e\}$ multiplies the lamp at the cursor by $c$. Let $\pi$ be the cursor homomorphism, whose kernel is the base group $B_F=\bigoplus_{\Z}F$. Set $h(u)=-u\log u-(1-u)\log(1-u)+u\log r$ for $u\in[0,1]$ (natural logarithms, $h(0)=0$, $h(1)=\log r$), and let
\[
\rho=\rho_r>1 \text{ solve } z^3=z+r,
\qquad
\omega=\omega_r=\tfrac{1+\sqrt{1+4r}}2 \text{ solve } z^2=z+r .
\]
Put $u_r=r/(r+\rho)\in(0,1)$ and
\[
\beta^{*}=\frac{1}{1+u_r}=\frac{\rho^{3}}{\rho^{3}+r},
\qquad
\bar\beta=\frac{\omega^{2}}{\omega^{2}+r}.
\]
The notation $\omega$ is consistent with its use above: Corollary \ref{cor:growth} identifies $\omega_r$ with the growth rate of $L_F$.

\begin{theorem}[Spectrum]\label{thm:spectrum}
For every $\beta\in[-1,1]$ and every sequence of integers $j_n$ with $j_n/n\to\beta$, the limit
$I(\beta)=\lim_n\frac1n\log N_{j_n}(n)$ exists, is even and concave, and equals
\[
I(\beta)\;=\;
\begin{cases}
(1+|\beta|)\,\log\rho, & |\beta|\le\beta^{*},\\[2pt]
|\beta|\;h\!\Big(\dfrac{1-|\beta|}{|\beta|}\Big), & \beta^{*}\le|\beta|\le1 .
\end{cases}
\]
The same formula holds for the sphere-fiber counts $N_j(n)-N_j(n-1)$.
\end{theorem}

\begin{theorem}[Phase transition]\label{thm:transition}
$I$ is continuously differentiable on $(-1,1)$ and real-analytic off $\{\pm\beta^{*}\}$; it is affine on $[-\beta^{*},\beta^{*}]$ and strictly concave on $[\beta^{*},1]$, and its second derivative jumps at $\beta^{*}$ from $0$ to $h''(u_r)/(\beta^{*})^{3}<0$: the transition is of second order. The maximum of $I$ is $\log\omega$, attained exactly at $\pm\bar\beta$, and $\beta^{*}<\bar\beta$.
\end{theorem}

Write an element as $x=(\eta,j)$ with lamp configuration $\eta$ and cursor $j$, and let $[a,b]$ be the smallest interval containing $\supp\eta\cup\{0,j\}$, of length $L=b-a$; the quantity $L-|j|\ge0$ measures the backtracking of a geodesic representative. Following the usage for random walks on these groups \cite{LPP96,PSC02}, where ballistic motion has linear drift and no backtracking, we call the phase $|\beta|>\beta^{*}$ \emph{ballistic}, since there $L/n\to|\beta|$, and the phase $|\beta|<\beta^{*}$ \emph{subballistic}, since there $L/n\to(1+|\beta|)/(2+u_r)>|\beta|$. Theorems \ref{thm:shape} and \ref{thm:rigidity} make this precise; on the subballistic phase the conditioned lamp density is $u_r$, independent of $\beta$, while on the ballistic phase it is $(1-|\beta|)/|\beta|$.

\begin{theorem}[Limit shape]\label{thm:shape}
Let $Z_n$ be uniform on $B(n)$, with cursor $\pi(Z_n)$, hull length $L_n$, and lamp count $m_n$. Then, with exponentially small exception probabilities,
\[
\frac{|\pi(Z_n)|}{n}\ \longrightarrow\ \bar\beta,
\qquad
\frac{L_n}{n}\ \longrightarrow\ \bar\beta,
\qquad
\frac{m_n}{n}\ \longrightarrow\ 1-\bar\beta ,
\]
in probability; in particular $(L_n-|\pi(Z_n)|)/n\to0$, so a typical element of the ball is ballistic. Moreover $\pi(Z_n)/n$ satisfies a large deviation principle with speed $n$ and rate function $J=\log\omega-I$, which vanishes exactly at $\pm\bar\beta$ and is affine on $[-\beta^{*},\beta^{*}]$.
\end{theorem}

\begin{corollary}[Concentration]\label{cor:concentration}
$J''(\bar\beta)=(1+4r)^{3/2}/r$; consequently, for every $\varepsilon>0$ there are $\delta,C>0$ such that for all $0<s\le\delta$ and all $n$,
\[
\Prob\Big(\Big|\tfrac{|\pi(Z_n)|}{n}-\bar\beta\Big|\ge s\Big)\ \le\ C\,n^{3}\exp\Big(-n\,s^{2}\Big(\tfrac{(1+4r)^{3/2}}{2r}-\varepsilon\Big)\Big).
\]
For $r=1$ the Gaussian constant is $(1+4r)^{3/2}/(2r)=\tfrac{5\sqrt5}{2}$.
\end{corollary}

The exact length formula of Lemma \ref{lem:length} also makes the full bivariate growth series rational, in a form whose singularities separate the two phases.

\begin{theorem}[Bivariate growth series]\label{thm:series}
In $\Z[[s]][y,y^{-1}]$, and as an identity of meromorphic functions on the domain described in the proof,
\[
W(s,y)\;:=\;\sum_{x\in L_F}s^{|x|}y^{\pi(x)}
\;=\;P(s)^{2}\,(1+rs)\sum_{j\in\Z}\big(s(1+rs)\big)^{|j|}y^{\,j},
\qquad
P(s)=1+\frac{r s^{3}}{1-s^{2}(1+rs)} .
\]
The right-hand side continues meromorphically in $s$, for fixed $y>0$, past its domain of convergence; its polar set consists of the \emph{frozen} curve $s^{2}(1+rs)=1$, independent of $y$, whose positive root is $1/\rho$, and the \emph{moving} curves $s(1+rs)y^{\pm1}=1$, whose root at $y=1$ is $1/\omega$. The affine phase of $I$ is the Legendre transform of the frozen singularity, the ballistic phase that of the moving one, and $\beta^{*}$ marks the exchange of dominance.
\end{theorem}

Specializing to the coefficient of $y^{0}$ identifies the growth of the kernel: $B(n)\cap B_F$ consists of the ball elements whose cursor returns to the origin.

\begin{corollary}[Relative growth of the base group]\label{cor:gap}
The relative growth series of $B_F=\bigoplus_\Z F$ in $L_F$ is the rational function $(1+rs)P(s)^{2}$, and
\[
\lim_{n\to\infty}\ \#\big(B(n)\cap B_F\big)^{1/n}\;=\;\rho_r\;<\;\omega_r .
\]
Thus $B_F$ is a normal subgroup with amenable quotient $\Z$, hence co-amenable, yet it has a strict and explicitly computable growth deficiency $\log\omega_r-\log\rho_r>0$. Since $\omega_r\sim\sqrt r$ and $\rho_r\sim r^{1/3}$, this deficiency grows like $\tfrac16\log r$ and is unbounded over the family.
\end{corollary}

Corollary \ref{cor:gap} should be read against the growth-gap criteria available in negative curvature. Grigorchuk \cite{Grigorchuk80} and Cohen \cite{Cohen82} proved that a normal subgroup $N$ of a free group $F_k$ satisfies $\lim\#(N\cap B(n))^{1/n}=2k-1$ if and only if $F_k/N$ is amenable, and Coulon, Dal'bo and Sambusetti \cite{CDS18} proved the analogous equivalence, equal exponential growth rates if and only if co-amenable, for a subgroup of a Gromov hyperbolic group acting properly cocompactly on its Cayley graph or on a CAT($-1$) space. Their hypotheses exclude $L_F$, which is amenable and not hyperbolic, so Corollary \ref{cor:gap} contradicts none of these results; what it shows is that the equivalence does not extend beyond negative curvature, and that it fails already in the most classical amenable group of exponential growth, with an exactly computable deficiency: a cubic rate against a quadratic one. The directional spectrum locates the missing growth, since the fibers realizing the full rate are those over $\pm\bar\beta n$ rather than the fiber over $0$.

\begin{example}[The classical lamplighter]\label{ex:classical}
For $r=1$ the cubic and the quadratic are $z^{3}=z+1$ and $z^{2}=z+1$, so $\rho$ is the plastic number, the smallest Pisot number, and $\omega=\varphi$ is the golden ratio. The constants are
\[
\varphi=1.618034\ldots,\quad
\rho=1.324718\ldots,\quad
u_1=\tfrac1{1+\rho}=0.430160\ldots,
\]
\[
\beta^{*}=\tfrac{1+\rho}{2+\rho}=0.699223\ldots,
\qquad
\bar\beta=\tfrac{\varphi^{2}}{\varphi^{2}+1}=\tfrac{\varphi}{\sqrt5}=0.723607\ldots,
\qquad
\tfrac{1-\bar\beta}{\bar\beta}=\varphi^{-2}=0.381966\ldots,
\]
the identity $\bar\beta=\varphi/\sqrt5$ following from $\varphi^{2}=\varphi+1$ and $2\varphi-1=\sqrt5$. Thus a uniformly random element of $B(n)$ has cursor at distance $\approx n\varphi/\sqrt5$, hull length $\approx n\varphi/\sqrt5$ as well, so no macroscopic backtracking, and lamp density $\varphi^{-2}$ inside its hull; while conditioning the cursor to any height below $0.699\,n$ forces macroscopic backtracking and lamp density $0.430\ldots$, whatever the height.
\end{example}

Bucher and Talambutsa observe that the second largest growth rate of $L_2$ over its generating sets is the plastic number, and ask for a natural geometric reason for its appearance (\cite{BT17}, Remark 9, where it is also noted that $\rho$ is the minimal growth rate of $\mathrm{GL}_2(\Z)$). Corollary \ref{cor:gap} gives an occurrence of $\rho$ inside the \emph{standard} Cayley graph of $L_2$, as the exponential rate at which ball elements return the cursor to the origin. Whether the two occurrences are related is an open question.

The family of spectra also degenerates in a controlled way as the lamp group grows.

\begin{proposition}[Tropical limit]\label{prop:tropical}
As $r\to\infty$,
\[
\sup_{\beta\in[-1,1]}\ \Big|\frac{I_r(\beta)}{\log r}-\min\Big(\frac{1+|\beta|}{3},\,1-|\beta|\Big)\Big|\;=\;O\Big(\frac1{\log r}\Big),
\]
and $\beta^{*}_r\to\tfrac12$, $\bar\beta_r\to\tfrac12$: in the large-lamp limit the spectrum converges to a piecewise-linear profile whose single kink absorbs both the transition and the peak.
\end{proposition}

The method is elementary throughout. That the growth series of $F\wr\Z$ is rational is classical \cite{Parry92,Bartholdi17}; Theorem \ref{thm:series} refines that series by the cursor variable, and the phase transition can then be read off its singularities.

\section{The exact length formula and fiber counts}\label{sec:exact}

Word length in $L_F$ is given by an exact formula, not merely up to bounded error, and everything below rests on this. In any wreath product with base group $Q$ and the standard generators, the length of an element is the total length of its lamp values plus the length of a shortest path in $\mathrm{Cay}(Q,S_Q)$ that starts at the identity, visits the support of the lamp configuration, and ends at the base coordinate; this is classical, going back to Parry \cite{Parry92} (see \cite{Bartholdi17}, and \cite{Bodart22} for a recent account). For $Q=\Z$ the traveling-salesman term is explicit, which is what makes the present computation possible. This section records the formula in that form (Lemma \ref{lem:length}), converts it into a closed-form count of each fiber by decomposing an element into a core and two hull extensions (Lemma \ref{lem:count}), and states the entropy estimates used to pass from counts to exponential rates.

Elements of $L_F=F\wr\Z$ are pairs $x=(\eta,j)$, where $\eta\colon\Z\to F$ is finitely supported and $j\in\Z$; multiplication is $(\eta,j)(\eta',j')=(\eta+\tau_j\eta',\,j+j')$ with $\tau_j\eta'(i)=\eta'(i-j)$. The generator $t=(0,1)$ moves the cursor, and $c\in F\setminus\{e\}$ is identified with $(\delta^{c},0)$, where $\delta^{c}$ is the configuration equal to $c$ at the origin and trivial elsewhere, so that $c$ multiplies the lamp at the cursor by $c$. Write $m=m(x)=|\supp\eta|$, let $[a,b]$ be the smallest interval containing $\supp\eta\cup\{0,j\}$, and set $L=L(x)=b-a$. Throughout, the letter $i$ denotes a site of $\Z$; the variable $s$ is reserved for the generating function of Section \ref{sec:series}.

\begin{lemma}[Length formula]\label{lem:length}
For every $x=(\eta,j)$, $\ |x|=m+2L-|j|$.
\end{lemma}

This is the classical formula in the case $Q=\Z$, where the traveling-salesman term evaluates to $2L-|j|$; the proof is included to fix the notation used throughout.

\begin{proof}
\emph{Lower bound.} Fix a word of length $\ell$ representing $x$ and follow the cursor: it traces a nearest-neighbor path on $\Z$ from $0$ to $j$. The final lamp value at a site $i$ is the product of the letters applied while the cursor stood at $i$, so if $\eta(i)\neq e$ then at least one lamp letter was applied at $i$; hence the path visits every site of $\supp\eta$, and in particular visits $a$ and $b$, each of which lies in $\supp\eta\cup\{0,j\}$. A path from $0$ to $j$ visiting $a$ and $b$ has length at least
$\min\big(|a|+(b-a)+|b-j|,\ |b|+(b-a)+|j-a|\big)$,
according to which extreme is visited first. For $j\ge0$, where $a\le0\le j\le b$, the two options are $-a+L+(b-j)=2L-j$ and $b+L+(j-a)=2L+j$, so the minimum is $2L-j$; the case $j\le0$ is the mirror image, giving $2L+j$. The word also contains at least $m$ lamp letters, one for each site of $\supp\eta$, whence $\ell\ge m+2L-|j|$.

\emph{Upper bound.} Assume $j\ge0$, the case $j\le0$ being the mirror image. Walk $0\to a$, then $a\to b$, then $b\to j$, applying at each site $i\in\supp\eta$, on its first visit, the single lamp letter $\eta(i)\in F\setminus\{e\}$. This word represents $x$ and uses $-a+(b-a)+(b-j)=2L-j$ moves together with $m$ lamp letters. The construction covers $j=0$, where the path is $0\to a\to b\to0$ of length $2L$, and the degenerate case $L=0$, where the path is empty and $|x|=m\in\{0,1\}$ according as $x=e$ or $x$ is a single lamp letter at the origin.
\end{proof}

Lemma \ref{lem:length} makes the fiber counts computable in closed form. For $j\ge0$ record each $x$ in the fiber by its hull extensions $e_-=\min(0,j)-a\ge0$ and $e_+=b-\max(0,j)\ge0$, so that $L=j+e_-+e_+$; if $e_->0$ then the site $a$ carries a nontrivial lamp, since otherwise the hull would be smaller, and similarly for $e_+$.

\begin{lemma}[Exact fiber count]\label{lem:count}
For $0\le j\le n$,
\[
N_j(n)\;=\;\sum_{e_-,e_+\ge0}\ r^{f}\!\!\sum_{k=0}^{K(e_-,e_+)}\binom{L+1-f}{k}r^{k},
\]
where $L=j+e_-+e_+$, $f=\mathbf 1_{\{e_->0\}}+\mathbf 1_{\{e_+>0\}}$ and
$K=\min\big(L+1-f,\ n-(2L-j)-f\big)$, with the convention that a sum with negative upper limit is $0$.
\end{lemma}

\begin{proof}
By Lemma \ref{lem:length}, $x=(\eta,j)$ lies in $B(n)$ if and only if $m+2L-j\le n$. Given $(e_-,e_+)$ the hull is determined; the $f$ forced endpoint lamps take any of $r$ nontrivial values each; the remaining $L+1-f$ sites carry an arbitrary configuration with $k$ nontrivial lamps of $r$ values each; and the length constraint reads $k+f+2L-j\le n$, that is, $k\le K$.
\end{proof}

The standard entropy estimates read, for integers $0\le k\le M$,
\begin{equation}\label{eq:entropy}
\frac{e^{M h(k/M)}}{M+1}\ \le\ \binom Mk r^{k}\ \le\ e^{M h(k/M)} ,
\end{equation}
by the bound $\binom Mk\le 2^{M H_2(k/M)}\le(M+1)\binom Mk$, where $H_2$ is the binary entropy.

\section{The variational problem}\label{sec:variational}

Lemma \ref{lem:count} expresses each fiber count as a sum of $O(n^{2})$ binomial terms, so its exponential rate is the maximum of an entropy functional over a two-dimensional polytope of rescaled parameters, the hull length and the lamp count. This section solves that maximization in closed form. The budget constraint is always active, which reduces the problem to one variable, and the answer then depends on whether the unconstrained maximizer of an auxiliary function $\psi$ lies inside the admissible range. That dichotomy is the origin of the phase transition, and it is also where the two critical exponents appear: the interior maximizer forces the cubic $\rho^{3}=\rho+r$, while the boundary regime later forces the quadratic $\omega^{2}=\omega+r$.

For $\beta\in[0,1]$ let
\[
\Delta_\beta=\Big\{(\ell,\mu):\ \ell\ge\beta,\ \ 0\le\mu\le\ell,\ \ \mu+2\ell-\beta\le1\Big\},
\qquad
\Psi(\beta)=\sup_{(\ell,\mu)\in \Delta_\beta}\ \ell\, h(\mu/\ell),
\]
with the convention $\ell h(\mu/\ell)=0$ when $\ell=0$; the parameters correspond to $L\approx\ell n$ and $m\approx\mu n$. The set $\Delta_\beta$ is a nonempty compact polytope depending continuously on $\beta$, and $(\ell,\mu)\mapsto\ell h(\mu/\ell)$ is continuous on $\{0\le\mu\le\ell\le1\}$, so $\Psi$ is continuous on $[0,1]$.

\begin{proposition}\label{prop:variational}
For every $\beta\in[0,1]$, $\ \Psi(\beta)=I(\beta)$ as defined in Theorem \ref{thm:spectrum}, the maximizer is unique, and it equals
\[
(\ell_\beta,\mu_\beta)=
\Big(\tfrac{1+\beta}{2+u_r},\ \tfrac{(1+\beta)u_r}{2+u_r}\Big)\ \ (\beta\le\beta^{*}),
\qquad
(\ell_\beta,\mu_\beta)=\big(\beta,\ 1-\beta\big)\ \ (\beta\ge\beta^{*}).
\]
\end{proposition}

\begin{proof}
\emph{Positivity.} Let $\beta\in[0,1)$ and pick any $u\in(0,u_{\max}(\beta)]$, where $u_{\max}(\beta)=\min\big(1,(1-\beta)/\beta\big)$ and $u_{\max}(0):=1$. Setting $\ell=(1+\beta)/(2+u)$ and $\mu=u\ell$ gives $\mu+2\ell-\beta=\ell(2+u)-\beta=1$ and $\ell\ge\beta$, so $(\ell,\mu)\in \Delta_\beta$ and $\Psi(\beta)\ge\ell h(u)>0$.

\emph{The budget binds.} For $0<u\le1$,
\begin{equation}\label{eq:partials}
\frac{\partial}{\partial\ell}\,\ell h(\mu/\ell)\Big|_{\mu}\;=\;h(u)-u h'(u)\;=\;-\log(1-u)\;>\;0,
\qquad u=\mu/\ell ,
\end{equation}
interpreted as $+\infty$ at $u=1$. Let $(\ell,\mu)$ maximize. If $\mu=0$ the value is $0$, which by positivity is impossible for $\beta<1$; and for $\beta=1$ the constraints $\ell\ge1$ and $\mu+2\ell\le2$ force $\ell=1$, $\mu=0$, where the budget binds. So assume $\mu>0$. If $\mu+2\ell-\beta<1$, increasing $\ell$ slightly keeps $\ell\ge\beta$ and $\mu\le\ell$, keeps the budget satisfied, and strictly increases the value by \eqref{eq:partials}, a contradiction. Hence $\mu+2\ell-\beta=1$ at every maximizer.

\emph{Reduction to one variable.} On the budget face, writing $u=\mu/\ell\in[0,1]$ gives $\ell(2+u)=1+\beta$, so the constraint $\ell\ge\beta$ reads $\beta(2+u)\le1+\beta$, that is $u\le(1-\beta)/\beta$, vacuous at $\beta=0$. Therefore
\[
\Psi(\beta)=(1+\beta)\,\max_{0\le u\le u_{\max}(\beta)}\psi(u),
\qquad
\psi(u)=\frac{h(u)}{2+u}.
\]

\emph{The unconstrained maximum of $\psi$.} Let $g(u)=(2+u)h'(u)-h(u)$, so $\psi'(u)=g(u)/(2+u)^2$. Then $g'(u)=(2+u)h''(u)<0$, while $h'(u)=\log\frac{r(1-u)}{u}$ gives $g(0^{+})=+\infty$ and $g(1^{-})=-\infty$. By the intermediate value theorem $g$ has a zero, unique since $g$ is strictly decreasing; call it $u_r\in(0,1)$. Thus $\psi$ increases on $[0,u_r]$ and decreases on $[u_r,1]$, and
\begin{equation}\label{eq:critical}
\kappa:=\psi(u_r)=h'(u_r).
\end{equation}
Set $\rho=e^{\kappa}$. From $h'(u)=\log\frac{r(1-u)}{u}$, equation \eqref{eq:critical} gives $r(1-u_r)/u_r=\rho$, that is $u_r=r/(r+\rho)$, and then $1-u_r=\rho/(r+\rho)$. Substituting into $h$,
\[
h(u_r)=u_r\log\tfrac{r+\rho}{r}+(1-u_r)\log\tfrac{r+\rho}{\rho}+u_r\log r
=\log(r+\rho)-(1-u_r)\log\rho ,
\]
while \eqref{eq:critical} reads $h(u_r)=\kappa(2+u_r)=(2+u_r)\log\rho$. Equating the two and cancelling,
\[
\log(r+\rho)=(2+u_r)\log\rho+(1-u_r)\log\rho=3\log\rho,
\]
that is
\begin{equation}\label{eq:cubic}
\rho^{3}=\rho+r .
\end{equation}
Since $z\mapsto z^{3}-z$ is strictly increasing for $z\ge1$ and vanishes at $z=1$, the root $\rho=\rho_r>1$ of \eqref{eq:cubic} is unique, and $\kappa=\log\rho_r$.

\emph{The two phases.} If $u_r\le u_{\max}(\beta)$, equivalently $\beta\le1/(1+u_r)=\beta^{*}$, the maximum is interior: $\Psi(\beta)=(1+\beta)\kappa=(1+\beta)\log\rho$, with unique maximizer $u=u_r$ and $\ell_\beta=(1+\beta)/(2+u_r)$. If $\beta>\beta^{*}$, the constraint binds, and since $\psi$ is strictly increasing on $[0,u_r]$ the unique maximizer is $u=(1-\beta)/\beta$; there $2+u=(1+\beta)/\beta$, so
\[
\Psi(\beta)=(1+\beta)\,\frac{h(u)}{(1+\beta)/\beta}=\beta\,h\Big(\tfrac{1-\beta}\beta\Big),
\qquad
\ell_\beta=\beta,\quad \mu_\beta=1-\beta .
\]
Finally $\beta^{*}=1/(1+u_r)=(r+\rho)/(2r+\rho)=\rho^{3}/(\rho^{3}+r)$ by \eqref{eq:cubic}. Uniqueness of the maximizer follows from the strict unimodality of $\psi$ together with \eqref{eq:partials}.
\end{proof}

\begin{proof}[Proof of Theorem \ref{thm:spectrum}]
Evenness holds because the automorphism inverting $t$ and fixing $F$ maps the fiber over $j$ isometrically onto the fiber over $-j$; so assume $j_n\ge0$ and $j_n/n\to\beta$.

\emph{Upper bound.} In Lemma \ref{lem:count} every feasible pair satisfies $2L-j_n\le n$, so there are at most $(n+1)^2$ pairs $(e_-,e_+)$ and at most $n+2$ values of $k$, and by \eqref{eq:entropy} each term is at most $r^{2}e^{(L+1-f)h(k/(L+1-f))}$. The point $\big(\tfrac{L+1-f}{n},\tfrac{k+f}{n}\big)$ lies within $O(1/n)$ of $P_{j_n/n}$; by uniform continuity of $(\ell,\mu)\mapsto\ell h(\mu/\ell)$ on $\{0\le\mu\le\ell\le2\}$ and continuity of $\beta\mapsto \Delta_\beta$, given $\varepsilon>0$ every exponent is at most $n(\Psi(\beta)+\varepsilon)$ for $n$ large. Hence $N_{j_n}(n)\le r^{2}(n+1)^{2}(n+2)e^{n(\Psi(\beta)+\varepsilon)}$, so $\limsup\frac1n\log N_{j_n}(n)\le\Psi(\beta)$.

\emph{Lower bound.} Let $(\ell_\beta,\mu_\beta)$ be the maximizer and fix $\varepsilon>0$. Take $e_-=0$, $e_+=\max(0,\lceil\ell_\beta n\rceil-j_n)$ and $k=\lfloor\mu_\beta n\rfloor-f-3$, which is nonnegative for large $n$ because $\mu_\beta>0$ when $\beta<1$; the case $\beta=1$ is trivial since $I(1)=0$ and $N_{j_n}(n)\ge1$. For large $n$ the constraint $k+f+2L-j_n\le n$ holds, since $\mu_\beta+2\ell_\beta-\beta\le1$, the roundings cost at most $5$, which the subtracted $3$ and the slack from $j_n=\beta n+o(n)$ absorb. By \eqref{eq:entropy} and continuity, the single term of Lemma \ref{lem:count} indexed by this choice is at least $e^{n(\Psi(\beta)-\varepsilon)}$ for $n$ large, so $\liminf\frac1n\log N_{j_n}(n)\ge\Psi(\beta)$. Proposition \ref{prop:variational} completes the proof for balls.

For spheres, increase $k$ by at most $2$ so that $k+f+2L-j_n=n$ exactly, which changes the exponent by $O(1)$; the resulting elements have length exactly $n$, giving the same lower bound, and the upper bound is inherited from the ball counts.
\end{proof}

\begin{proof}[Proof of Theorem \ref{thm:transition}]
On $[0,\beta^{*}]$ the function $I$ is affine with slope $\kappa=\log\rho$. On $(\beta^{*},1)$ put $u=u(\beta)=(1-\beta)/\beta$, so that $u'=-1/\beta^{2}$ and $I(\beta)=\beta h(u)$; then
\[
I'(\beta)=h(u)-\frac{h'(u)}{\beta},
\qquad
I''(\beta)=h'(u)u'-\Big(\frac{h''(u)u'}{\beta}-\frac{h'(u)}{\beta^{2}}\Big)
=\frac{h''(u)}{\beta^{3}}\;<\;0 ,
\]
since $h''(u)=-1/(u(1-u))<0$, so this branch is strictly concave and real-analytic. At $\beta^{*}$ one has $u=u_r$ and $1/\beta^{*}=1+u_r$, so \eqref{eq:critical} gives
\[
I'(\beta^{*+})=h(u_r)-(1+u_r)h'(u_r)=\kappa(2+u_r)-(1+u_r)\kappa=\kappa ,
\]
matching the affine slope, so $I\in C^{1}$, while $I''$ jumps from $0$ to $h''(u_r)/(\beta^{*})^{3}\neq0$. Concavity on $[-1,1]$ follows because $I'$ is continuous, constant on $[0,\beta^{*}]$, decreasing thereafter, and $I$ is even.

\emph{The peak.} On the strictly concave branch the equation $I'(\bar\beta)=0$ reads $h(\bar u)=(1+\bar u)h'(\bar u)$ with $\bar u=(1-\bar\beta)/\bar\beta$. Since $I'(\beta^{*+})=\kappa>0$ and $I'(\beta)\to-\infty$ as $\beta\to1^{-}$, because $h'(u)\to+\infty$ as $u\to0^{+}$, such a $\bar\beta\in(\beta^{*},1)$ exists and is unique. Set $\omega=e^{h'(\bar u)}$, so that $r(1-\bar u)/\bar u=\omega$ and $\bar u=r/(r+\omega)$, $1-\bar u=\omega/(r+\omega)$. As in Proposition \ref{prop:variational},
\[
h(\bar u)=\bar u\log\tfrac{r+\omega}{r}+(1-\bar u)\log\tfrac{r+\omega}{\omega}+\bar u\log r=\log(r+\omega)-(1-\bar u)\log\omega ,
\]
so the equation $h(\bar u)=(1+\bar u)\log\omega$ becomes
\[
\log(r+\omega)-(1-\bar u)\log\omega=(1+\bar u)\log\omega,
\qquad\text{i.e.}\qquad
\log(r+\omega)=2\log\omega,
\]
that is
\begin{equation}\label{eq:quadratic}
\omega^{2}=\omega+r,
\end{equation}
so $\omega=\omega_r$. Since $\bar\beta(1+\bar u)=1$, it follows that $I(\bar\beta)=\bar\beta h(\bar u)=\bar\beta(1+\bar u)h'(\bar u)=h'(\bar u)=\log\omega$, and $\bar\beta=1/(1+\bar u)=(r+\omega)/(2r+\omega)=\omega^{2}/(\omega^{2}+r)$ by \eqref{eq:quadratic}. Finally $\omega>\rho$, because evaluating $z^{3}-z-r$ at $\omega$ gives $\omega^{3}-\omega-r=\omega^{3}-\omega^{2}=\omega^{2}(\omega-1)>0$ while the cubic is increasing at its root; hence
$\bar\beta=\omega^{2}/(\omega^{2}+r)>\rho^{3}/(\rho^{3}+r)=\beta^{*}$.
\end{proof}

\begin{corollary}[Growth rate, localized]\label{cor:growth}
$\lim_n V(n)^{1/n}=\omega_r=\tfrac{1+\sqrt{1+4r}}2$, and the fibers realizing the growth are exactly those with $j/n\to\pm\bar\beta$.
\end{corollary}

\begin{proof}
$\max_j N_j(n)\le V(n)\le(2n+1)\max_j N_j(n)$, and by Theorem \ref{thm:spectrum} with compactness and continuity of $I$, $\frac1n\log\max_jN_j(n)\to\max_\beta I=\log\omega$, attained only at $\pm\bar\beta$ by Theorem \ref{thm:transition}.
\end{proof}

For $r=1$ this recovers the golden-ratio growth of the classical lamplighter \cite{Parry92,LPP96}; the content added is the localization at $\pm\bar\beta$.

\section{The bivariate growth series}\label{sec:series}

The decomposition of an element into a core and two hull extensions was used in Lemma \ref{lem:count} one fiber at a time. Summing it over all fibers simultaneously, with a variable $y$ marking the cursor position, produces a closed rational form for the bivariate growth series. This yields a second and purely analytic account of the phase transition, as a competition between two families of singularities of which one does not move with $y$, and it identifies the relative growth series of the base group, proving Corollary \ref{cor:gap}.

\begin{proof}[Proof of Theorem \ref{thm:series}]
Fix $x=(\eta,j)$ with $j\ge0$ and split the hull $[a,b]$ into the left extension block of $e_-$ sites, the core $[0,j]$ of $j+1$ sites, and the right extension block of $e_+$ sites. By Lemma \ref{lem:length} the weight factorizes,
\[
s^{|x|}y^{\pi(x)}
=\big(s^{2e_-}\cdot\text{left-block lamps}\big)
\cdot\big(s^{j}y^{j}\cdot\text{core lamps}\big)
\cdot\big(s^{2e_+}\cdot\text{right-block lamps}\big),
\]
each nontrivial lamp contributing $rs$, namely a factor $s$ from the term $m$ in the length together with $r$ color choices, and each trivial site contributing $1$. A block with $e_\pm=E\ge1$ has its outermost site forced nontrivial and $E-1$ free sites, hence generating function $rs^{2E+1}(1+rs)^{E-1}$; summing the resulting geometric series,
\[
P(s)=1+\sum_{E\ge1}r\,s^{2E+1}(1+rs)^{E-1}=1+rs^{3}\sum_{k\ge0}\big(s^{2}(1+rs)\big)^{k}=1+\frac{rs^{3}}{1-s^{2}(1+rs)} .
\]
The core contributes $(1+rs)^{j+1}s^{j}y^{j}$. Multiplying the three factors and summing over $j\ge0$, then adding the mirror image for $j<0$ with $j=0$ counted once, yields the stated product.

All rearrangements are identities in $\Z[[s]][y,y^{-1}]$, since for fixed $n$ only finitely many terms contribute to the coefficient of $s^{n}$. Analytically, fix $y>0$; for $|s|$ small enough that $|s^{2}(1+rs)|<1$ and $|s(1+rs)|\max(y,y^{-1})<1$, both geometric series converge absolutely and the identity holds between analytic functions. The right-hand side is a rational function of $s$ and $y$, hence continues meromorphically in $s$ to all of $\C$, and the singularity analysis below is carried out on this continuation. Its poles are the zeros of $1-s^{2}(1+rs)$, contributed by $P$, and of $1-s(1+rs)y^{\pm1}$, contributed by the two geometric sums over $j$. Substituting $z=1/s$ into $s^{2}(1+rs)=1$ gives $z^{3}=z+r$, so the positive frozen root is $1/\rho$; substituting into $s(1+rs)=1$, the moving curve at $y=1$, gives $z^{2}=z+r$, with root $1/\omega$.

On the affine phase the fiber exponent $(1+|\beta|)\log\rho$ is linear in $\beta$ with slope $\log\rho$ determined by the frozen singularity alone, which does not move with $y$; a linear stretch of this kind is the standard signature of a $y$-independent dominant singularity in a Legendre transform. On the ballistic phase the exponent is the Legendre transform of the moving singularity and is strictly concave. The exchange of dominance occurs at $\beta^{*}$.
\end{proof}

\begin{proof}[Proof of Corollary \ref{cor:gap}]
The coefficient of $y^{0}$ in Theorem \ref{thm:series} is $(1+rs)P(s)^{2}$, a rational function whose coefficient of $s^{n}$ counts $\{x\in B_F:|x|=n\}$, so the relative growth series of $B_F$ is rational; dividing by $1-s$ gives the corresponding ball series, also rational. The exponential rate is the case $\beta=0$ of Theorem \ref{thm:spectrum}, namely $I(0)=\log\rho_r$, and $\omega>\rho$ was shown in the proof of Theorem \ref{thm:transition}. Since $L_F/B_F\cong\Z$ is amenable, $B_F$ is co-amenable.
\end{proof}

\begin{remark}\label{rem:mechanism}
The transition is visible from two sides. Combinatorially, producing an element over $\beta n$ costs $2\ell-\beta$ in cursor moves and $\mu$ in lamp letters per unit length and buys entropy $\ell h(\mu/\ell)$; by \eqref{eq:partials} the marginal entropy of one further site exceeds its cost even when the site lies beyond the target, as long as $|\beta|$ is small, so the optimal element overshoots and returns, with lamp density locked at the budget-free value $u_r$. At $\beta^{*}$ the cursor constraint absorbs the entire budget and overshooting becomes unaffordable, which is why the density must then fall to $(1-|\beta|)/|\beta|$. Analytically, the fiber exponent is governed by whichever singularity of $W(s,y)$ is nearest the origin, and the frozen one, being independent of $y$, produces the affine stretch.
\end{remark}

\section{Limit shape, concentration, large deviations}\label{sec:shape}

The variational problem of Section \ref{sec:variational} has a unique maximizer for each $\beta$, and a unique global maximizer once $\beta$ is also free. Uniqueness is what converts counting estimates into probabilistic statements: terms whose rescaled parameters stay away from the maximizer are exponentially negligible against the total, so a uniformly random element of the ball concentrates on the optimal profile. This section carries out that argument three times, giving the limit shape and large deviation principle of Theorem \ref{thm:shape}, the Gaussian constant of Corollary \ref{cor:concentration}, and the rigidity of the conditioned measure in the subballistic phase.

\begin{proof}[Proof of Theorem \ref{thm:shape}]
Consider $\Delta=\{(\beta,\ell,\mu):0\le\beta\le1,\ (\ell,\mu)\in\Delta_\beta\}$ and $\Phi(\beta,\ell,\mu)=\ell h(\mu/\ell)$ on it. By Proposition \ref{prop:variational} and Theorem \ref{thm:transition}, $\sup_\Delta\Phi=\log\omega$, attained at the single point $q^{*}=(\bar\beta,\bar\beta,1-\bar\beta)$: the maximum of $\Psi$ over $\beta$ is attained only at $\bar\beta$, and there the inner maximizer is unique. For $\varepsilon>0$ let $U_\varepsilon$ denote the open $\varepsilon$-neighborhood of $q^{*}$; by compactness and continuity, $\sup_{\Delta\setminus U_\varepsilon}\Phi\le\log\omega-\delta(\varepsilon)$ with $\delta(\varepsilon)>0$.

Each element of $B(n)$ contributes to exactly one term of the sum of Lemma \ref{lem:count}, taken jointly over $j$. By \eqref{eq:entropy} the total mass of the terms whose rescaled parameters lie outside $U_\varepsilon$ is at most $\mathrm{poly}(n)e^{n(\log\omega-\delta+o(1))}$, while $V(n)\ge e^{n(\log\omega-\delta/2)}$ for large $n$ by Corollary \ref{cor:growth}. Hence
\[
\Prob\Big(\big(\tfrac{|\pi(Z_n)|}n,\tfrac{L_n}n,\tfrac{m_n}n\big)\notin U_\varepsilon\Big)\ \le\ \mathrm{poly}(n)\,e^{-n\delta/2}.
\]
This is the stated concentration, and $L_n-|\pi(Z_n)|=o(n)$ because the $\ell$- and $\beta$-coordinates of $q^{*}$ agree.

\emph{Large deviations.} Write $\nu_n$ for the law of $\pi(Z_n)/n$ and $J=\log\omega-I$, a continuous nonnegative function by Theorems \ref{thm:spectrum} and \ref{thm:transition}. For closed $A\subset[-1,1]$,
\[
\nu_n(A)\ \le\ \frac{(2n+1)\max\{N_j(n):\ j/n\in A\}}{V(n)} .
\]
Choose $j_n$ with $j_n/n\in A$ realizing the maximum; by compactness a subsequence has $j_n/n\to\beta_0\in A$, and along it $\frac1n\log N_{j_n}(n)\to I(\beta_0)\le\sup_{A}I$ by Theorem \ref{thm:spectrum}. With Corollary \ref{cor:growth} this gives
$\limsup\frac1n\log\nu_n(A)\le\sup_A I-\log\omega=-\inf_A J$.
Now let $G\subset[-1,1]$ be open and $\beta\in G$. Choose integers $j_n$ with $j_n/n\to\beta$ and $j_n/n\in G$ for large $n$. Then
\[
\liminf_n\tfrac1n\log\nu_n(G)\ \ge\ \liminf_n\tfrac1n\log\big(N_{j_n}(n)/V(n)\big)\ =\ I(\beta)-\log\omega\ =\ -J(\beta),
\]
and taking the supremum over $\beta\in G$ gives $\liminf\frac1n\log\nu_n(G)\ge-\inf_G J$. Since $[-1,1]$ is compact and $J$ is continuous with compact sublevel sets, this is a full large deviation principle with speed $n$ and good rate function $J$, which vanishes exactly at $\pm\bar\beta$ and is affine on $[-\beta^{*},\beta^{*}]$.
\end{proof}

\begin{proof}[Proof of Corollary \ref{cor:concentration}]
On the ballistic branch $J$ is real-analytic with $J(\bar\beta)=J'(\bar\beta)=0$ and $J''(\bar\beta)=-I''(\bar\beta)=-h''(\bar u)/\bar\beta^{3}$. Now $h''(u)=-1/(u(1-u))$ and, using $r+\omega=\omega^{2}$,
\[
\bar u(1-\bar u)=\frac{r}{r+\omega}\cdot\frac{\omega}{r+\omega}=\frac{r\omega}{\omega^{4}}=\frac r{\omega^{3}},
\qquad
\bar\beta=\frac{\omega^{2}}{\omega^{2}+r}=\frac{\omega^{2}}{2\omega^{2}-\omega}=\frac{\omega}{2\omega-1},
\]
so $J''(\bar\beta)=\frac{\omega^{3}}{r}\cdot\frac{(2\omega-1)^{3}}{\omega^{3}}=\frac{(2\omega-1)^{3}}r=\frac{(1+4r)^{3/2}}r$, since $2\omega-1=\sqrt{1+4r}$. By Taylor's theorem, given $\varepsilon>0$ there is $\delta>0$ with $J(\beta)\ge\big(\tfrac{(1+4r)^{3/2}}{2r}-\varepsilon\big)(|\beta|-\bar\beta)^{2}$ whenever $\big||\beta|-\bar\beta\big|\le\delta$; outside that range $J\ge\min\big(J(\bar\beta-\delta),J(\bar\beta+\delta)\big)$, which after shrinking $\delta$ dominates the same quadratic for $s\le\delta$. Applying the large deviation upper bound of Theorem \ref{thm:shape} to the closed set $\{\beta:\ ||\beta|-\bar\beta|\ge s\}$, with the polynomial prefactors of that proof made explicit, gives the stated inequality.
\end{proof}

\begin{theorem}[Subballistic rigidity]\label{thm:rigidity}
Fix $|\beta|<\beta^{*}$ and condition $Z_n$ on $\pi(Z_n)=\lfloor\beta n\rfloor$. Then
\[
\frac{L_n}{n}\ \longrightarrow\ \frac{1+|\beta|}{2+u_r}\ >\ |\beta|,
\qquad
\frac{m_n}{L_n}\ \longrightarrow\ u_r=\frac{r}{r+\rho_r},
\]
in probability, and the convergence is exponentially fast, with rate at least $\delta_\varepsilon(\beta)/2$ where
$\delta_\varepsilon(\beta)=\Psi(\beta)-\sup\{\ell h(\mu/\ell):(\ell,\mu)\in \Delta_\beta,\ |(\ell,\mu)-(\ell_\beta,\mu_\beta)|\ge\varepsilon\}>0$.
Thus throughout the subballistic phase the conditioned element backtracks by the macroscopic amount $\big(\tfrac{1+|\beta|}{2+u_r}-|\beta|\big)n$, with lamp density $u_r$ independent of $\beta$. For $|\beta|>\beta^{*}$ instead $L_n/n\to|\beta|$ and $m_n/L_n\to(1-|\beta|)/|\beta|$.
\end{theorem}

\begin{proof}
Localization within the fiber: by Proposition \ref{prop:variational} the maximizer of $\Phi$ on $\Delta_\beta$ is unique, so $\delta_\varepsilon(\beta)>0$ by compactness and continuity, and by \eqref{eq:entropy} the terms of Lemma \ref{lem:count} whose rescaled parameters are $\varepsilon$-far from $(\ell_\beta,\mu_\beta)$ contribute at most $\mathrm{poly}(n)e^{n(\Psi(\beta)-\delta_\varepsilon(\beta))}$, against $N_{\lfloor\beta n\rfloor}(n)\ge e^{n(\Psi(\beta)-\delta_\varepsilon(\beta)/2)}$ for $n$ large.
\end{proof}

\begin{proof}[Proof of Proposition \ref{prop:tropical}]
Write $T(\beta)=\min\big(\tfrac{1+|\beta|}3,1-|\beta|\big)$, whose kink is at $|\beta|=\tfrac12$, and note that $T$ is $1$-Lipschitz. From \eqref{eq:cubic}, $\rho_r=\big(r(1+\rho_r/r)\big)^{1/3}$, so $\log\rho_r=\tfrac13\log r+O(r^{-2/3})$, and $u_r=r/(r+\rho_r)=1-O(r^{-2/3})$, whence $\beta^{*}_r=1/(1+u_r)=\tfrac12+O(r^{-2/3})$; also $\bar\beta_r=\omega_r/(2\omega_r-1)\to\tfrac12$ since $\omega_r\to\infty$.

On the affine phase $|\beta|\le\beta^{*}_r$,
$I_r(\beta)/\log r=(1+|\beta|)\log\rho_r/\log r=\tfrac{1+|\beta|}3+O\big(r^{-2/3}/\log r\big)$.
On the ballistic phase $|\beta|\ge\beta^{*}_r$, writing $u=(1-|\beta|)/|\beta|$ and $h(u)=u\log r+H(u)$ with $H(u)=-u\log u-(1-u)\log(1-u)\in[0,\log2]$, we get $I_r(\beta)=|\beta|u\log r+|\beta|H(u)=(1-|\beta|)\log r+O(1)$, so $I_r(\beta)/\log r=1-|\beta|+O(1/\log r)$.

For $|\beta|$ outside the interval between $\beta^{*}_r$ and $\tfrac12$, the branch of $I_r$ and the branch of $T$ agree, and the two displays above give $|I_r(\beta)/\log r-T(\beta)|=O(1/\log r)$. For $|\beta|$ inside that interval, which has length $O(r^{-2/3})$, we have $T(\beta)=T(\tfrac12)+O(r^{-2/3})=\tfrac12+O(r^{-2/3})$ since $T$ is $1$-Lipschitz, while whichever branch formula applies evaluates at $|\beta|=\tfrac12+O(r^{-2/3})$ to $\tfrac12+O(1/\log r)$, both branches of the limit profile taking the value $\tfrac12$ at $\tfrac12$. Hence the difference is $O(1/\log r)$ there as well, and the bound is uniform.
\end{proof}

\section{Problems}\label{sec:questions}

Proposition \ref{prop:tropical} shows that as the lamp group grows the normalized spectrum degenerates to a piecewise-linear profile, with $\beta^{*}_r$ and $\bar\beta_r$ merging at $\tfrac12$. The limiting profile is not itself the spectrum of a group, and the honest limiting object is the wreath product with infinite lamp group.

\begin{question}\label{q:ZwrZ}
Determine the directional growth spectrum of $\Z\wr\Z$ with the standard generators. A lamp of value $c$ now costs $|c|$, so the binary entropy $h$ must be replaced by the exponent governing restricted integer compositions, while the group still has exponential growth. Does the spectrum again consist of exactly two phases, and what equations replace \eqref{eq:cubic} and \eqref{eq:quadratic}?
\end{question}

Among the conclusions of Theorem \ref{thm:transition}, the affine phase is the one that fails in every previously computed example. Two independent descriptions of its mechanism are available, and neither refers to the lamp structure: analytically, by Theorem \ref{thm:series}, it arises from a dominant singularity of $W(s,y)$ that does not move with $y$; combinatorially, by Remark \ref{rem:mechanism}, it arises whenever the marginal entropy of extending an element beyond its target exceeds the marginal cost of traveling there and back. Both suggest that the phenomenon is not peculiar to $L_F$.

\begin{question}\label{q:affine}
Characterize the pairs $(G,\pi)$ whose directional growth spectrum contains an affine piece. Is the presence of an affine phase, or the value of $\beta^{*}$, independent of the finite generating set?
\end{question}

Corollary \ref{cor:gap} produces a co-amenable subgroup whose growth rate is strictly smaller than that of the ambient group, which the criteria of Grigorchuk--Cohen \cite{Grigorchuk80,Cohen82} and Coulon--Dal'bo--Sambusetti \cite{CDS18} forbid in their respective settings. The deficiency $\log(\omega_r/\rho_r)$ is unbounded over the family $L_F$, so no universal bound on such gaps can hold for amenable groups; but a single family says little about how common the phenomenon is.

\begin{question}\label{q:gap}
Which amenable groups of exponential growth admit co-amenable subgroups of strictly smaller exponential growth rate? Is there an intrinsic description of the deficiency, comparable to the interpretation of the vanishing gap as co-amenability in negative curvature?
\end{question}

Near $\pm\bar\beta$ the spectrum is strictly concave and real-analytic, Corollary \ref{cor:concentration} supplies the curvature $(1+4r)^{3/2}/r$ of the rate function there, and Theorem \ref{thm:series} supplies a rational generating function in two variables. In the hyperbolic setting these are precisely the ingredients from which local limit theorems are extracted \cite{Sharp01}. Note that $\pi(Z_n)/n$ concentrates on the two points $\pm\bar\beta$, so any such statement must be made after conditioning on the sign of $\pi(Z_n)$.

\begin{question}\label{q:clt}
Conditioned on $\pi(Z_n)>0$, does $\pi(Z_n)$ satisfy a central limit theorem, and a local limit theorem, about $\bar\beta n$ with variance asymptotic to $n\,r(1+4r)^{-3/2}$?
\end{question}


\begin{thebibliography}{99}

\bibitem{Bartholdi17}
L. Bartholdi,
\emph{Growth of groups and wreath products},
in: Groups, Graphs and Random Walks, London Math. Soc. Lecture Note Ser. \textbf{436}, Cambridge Univ. Press, Cambridge, 2017, 1--76.

\bibitem{Bodart22}
C. Bodart,
\emph{Dead ends and rationality of complete growth series},
preprint, arXiv:2210.07868.

\bibitem{BT17}
M. Bucher, A. Talambutsa,
\emph{Minimal exponential growth rates of metabelian Baumslag--Solitar groups and lamplighter groups},
Groups Geom. Dyn. \textbf{11} (2017), 189--209.

\bibitem{Cohen82}
J. M. Cohen,
\emph{Cogrowth and amenability of discrete groups},
J. Funct. Anal. \textbf{48} (1982), 301--309.

\bibitem{CDS18}
R. Coulon, F. Dal'bo, A. Sambusetti,
\emph{Growth gap in hyperbolic groups and amenability},
Geom. Funct. Anal. \textbf{28} (2018), 1260--1320.

\bibitem{DLM12}
M. Duchin, S. Leli\`evre, C. Mooney,
\emph{The geometry of spheres in free abelian groups},
Geom. Dedicata \textbf{161} (2012), 169--187.

\bibitem{GTT20}
I. Gekhtman, S. J. Taylor, G. Tiozzo,
\emph{Counting problems in graph products and relatively hyperbolic groups},
Israel J. Math. \textbf{237} (2020), 311--371.

\bibitem{Grigorchuk80}
R. I. Grigorchuk,
\emph{Symmetrical random walks on discrete groups},
in: Multicomponent Random Systems, Adv. Probab. Related Topics \textbf{6}, Dekker, New York, 1980, 285--325.

\bibitem{Grigorchuk84}
R. I. Grigorchuk,
\emph{Degrees of growth of finitely generated groups and the theory of invariant means},
Izv. Akad. Nauk SSSR Ser. Mat. \textbf{48} (1984), 939--985.

\bibitem{Gromov81}
M. Gromov,
\emph{Groups of polynomial growth and expanding maps},
Inst. Hautes \'Etudes Sci. Publ. Math. \textbf{53} (1981), 53--73.

\bibitem{dlH00}
P. de la Harpe,
\emph{Topics in Geometric Group Theory},
Chicago Lectures in Math., Univ. of Chicago Press, Chicago, 2000.

\bibitem{LPP96}
R. Lyons, R. Pemantle, Y. Peres,
\emph{Random walks on the lamplighter group},
Ann. Probab. \textbf{24} (1996), 1993--2006.

\bibitem{Parry92}
W. Parry,
\emph{Growth series of some wreath products},
Trans. Amer. Math. Soc. \textbf{331} (1992), 751--759.

\bibitem{PSC02}
C. Pittet, L. Saloff-Coste,
\emph{On random walks on wreath products},
Ann. Probab. \textbf{30} (2002), 948--977.

\bibitem{Rivin10}
I. Rivin,
\emph{Growth in free groups (and other stories)---twelve years later},
Illinois J. Math. \textbf{54} (2010), 327--370.

\bibitem{Sharp01}
R. Sharp,
\emph{Local limit theorems for free groups},
Math. Ann. \textbf{321} (2001), 889--904.

\end{thebibliography}
\end{document}